\documentstyle[amsfonts,amssymb,amsthm,amsmath, color, 12pt]{article}

\title{Some properties defined by relative versions of star-covering properties II}
\author{Maddalena Bonanzinga, Davide Giacopello and Fortunato Maesano\\
\\
\small\emph{Dedicated to the memory of Mikhail (Misha) Matveev}
}
\date{}

\begin{document}

 \maketitle

\newtheorem{theorem}{Theorem} [section]
\newtheorem{corollary}{Corollary}[section]
\newtheorem{question}{Question}[section]
\newtheorem{example}{Example}[section]
\newtheorem{lemma}{Lemma}[section]
\newtheorem{proposition}{Proposition}[section]
\newtheorem{property}{Property}[section]
\newtheorem{definition}{Definition}[section]
\newtheorem{remark}{Remark}[section]
\newtheorem{problem}{Problem}[section]

\newcommand{\my}[1]{\textcolor{red}{\sf #1}}
\newcommand{\green}[1]{\textcolor{green}{\sf #1}}
\newcommand{\blu}[1]{\textcolor{blue}{\sf #1}}
\newcommand{\violet}[1]{\textcolor{violet}{\sf #1}}

\begin{abstract}
 In this paper we consider 
 some recent relative versions of Menger property  called set strongly star Menger and set star Menger properties and the corresponding Hurewicz-type properties. In particular, using \cite {BMae}, we "easily" prove that the set strong star Menger and set strong star Hurewicz properties are between countable compactness and the property of having countable extent. Also we show that the extent of a regular set star Menger or a set star Hurewicz space cannot exceed $\frak c$. Moreover, we construct (1) a consistent example of a set star Menger (set star Hurewicz) space which is not set strongly star Menger (set strongly star Hurewicz) and  show that (2) the product of a set star Menger (set star Hurewicz) space with a compact space need not be set star Menger (set star Hurewicz). In particular, (1) and (2)  answer to some questions posed by Ko\v{c}inac, Konca and Singh in  \cite{KKS-MS} and \cite{S-AM}.
 
 \end{abstract}



\section{Introduction}

Let ${\mathcal U}$ be a cover of a space $X$ and $A$ be a subset of $X$; the star of $A$ with respect to ${\mathcal U}$ is the set $st(A,{\mathcal U})=\bigcup\{U:U\in{\mathcal U}\;\hbox{and}\;U\cap A\neq\emptyset\}$. 
The star of a one-point set $\{x\}$ with respect to a cover ${\mathcal U}$ is denoted by $st(x,{\mathcal U})$. 

Recall that a space $X$ is 
		star compact, briefly SC
		(strongly star compact, briefly SSC)  if for every open cover $\mathcal{U}$ of the space $X$, there exists a
		finite subfamily $\mathcal V$ of $\mathcal U$	
		(resp., a finite subset $F$ of $X$)  such that $st(\bigcup{\mathcal V},{\mathcal U})=X$ (resp., $st(F,{\mathcal U})=X$) (see \cite{IK} where another terminology is used, and \cite{vDRRT});	$X$ is star Lindel\"of, briefly SL (strongly star Lindel\"of, briefly SSL) if for every open cover $\mathcal{U}$ of the space $X$, there exists  a countable subfamily $\mathcal V$ of $\mathcal U$ (resp., a countable subset $C$ of $X$)  such that $st(\bigcup{\mathcal V},{\mathcal U})=X$ (resp., $st(C,{\mathcal U})=X$) (see \cite{IK2} and \cite{IK3}, where different terminology is used).

\smallskip

In \cite{KS,KKS-MS} Ko\v cinac, Konca and Singh introduced the following relative versions of SC, SSC, SL and SSL properties.

\begin{definition}\rm\label{Def2}\cite{KKS-MS}\label{SC}
	A space $X$ is set star compact, briefly set SC (resp., set strongly star compact, briefly set SSC), if for every nonempty subset $A$ of $X$ and for every family $\mathcal U$ of open sets in X such that $\overline {A}\subseteq \bigcup{\mathcal U}$, there exists a finite
	subfamily $\mathcal V$ of $\mathcal U$ (resp., finite subset $F$ of $\overline A$) such that $st(\bigcup{\mathcal V},{\mathcal U})\supset A$
	(resp., $st(F,{\mathcal U})\supset A$).\footnote{Recently, the properties of Definition \ref{Def2} were studied in \cite{BMae}. Note that in \cite{BMae} there is a \underline{misprint} in the statment of the definition of "relatively$^*$ SSC" that the authors use to describe set SSC property: in particular, the authors write that the set "$F$ is a finite subset of $A$" instead of "$F$ is a finite subset of $\overline A$".}
\end{definition}	

Replacing "finite" with "countable" in Definition \ref{SC}, one obtains the classes of set star Lindel\"of (briefly set SL) and set strongly star Lindel\"of (briefly set SSL) spaces (see \cite{KS}).

\smallskip

In the following CC means countably compact.

\begin{proposition}\rm \cite[Proposition 2.2]{BMae}\label{prop}
In the class of Hausdorff spaces SSC, set SSC and CC are equivalent properties. 
\end{proposition}
We prove the following
\begin{proposition}\rm \label{regular set SC} 
	In the class of regular spaces set SC and CC are equivalent properties. 
\end{proposition}
\begin{proof} Of course, every CC space is set SC. Now, let $X$ a regular set SC space. By contradiction, assume there exists a closed and discrete subspace $D=\{x_n: n\in\omega\}$ of $X$. By regularity, there exists a disjoint family ${\cal U}=\{U_n : n \in\omega\}$ of open subsets of $X$ such that $x_n\in U_n$, for every $n\in\omega$.  Then $D\subseteq \bigcup{\mathcal U}$ but  for every finite subfamily $\mathcal V$ of $\mathcal U$, we have that
		$D\not\subset st(\bigcup{\mathcal V},{\mathcal U})$;
		a contradiction.
\end{proof} 
 For a space $X$, $e(X)=\sup\{|C|\,\,:\,\, C\hbox{ is a closed and discrete subset of }X\}$ and $c(X)=\sup\{|{\cal A}|\,\,:\,\, {\cal A}\hbox{ is a cellular family of }X\}$ are, respectively, the extent and the cellularity of $X$. One says that a space $X$ has the countable chain condition (briefly ccc) if $c(X)=\omega$.
\begin{proposition}\rm \cite[Proposition 3.1]{BMae} \label{Proposition 3.1}
In the class of $T_1$ spaces, 
set SSL spaces are exactly spaces having countable extent. 
\end{proposition}

\begin{proposition}\rm \cite[Corollary 3.3]{BMae} 
	Every ccc space is set SL. 
\end{proposition}
Recall that a space $X$ is 
Menger, briefly M, if for each sequence $({\cal U}_n : n \in \omega)$ of open covers of $X$ there exists a sequence $({\cal V}_n : n \in \omega)$  such that ${\mathcal V}_n$, $n\in\omega$, is a finite subset of ${\mathcal U}_n$ and $X=\bigcup_{n \in \omega}\bigcup{\mathcal V}_n$; $X$ is Hurewicz, briefly H, if for each sequence $({\cal U}_n : n \in \omega)$ of open covers of $X$ there exists a sequence $({\cal V}_n : n \in \omega)$  such that ${\mathcal V}_n$, $n\in\omega$, is a finite subset of ${\mathcal U}_n$ and for every $x\in X$, $x \in\bigcup{\mathcal V}_n$ for all but finitely many $n\in\omega$.

 In \cite{K, K1, BCK} star versions of Menger and Hurewicz properties called star Menger, strongly star Menger, star Hurewicz and strongly star Hurewicz properties (Definitions \ref{star Menger} and \ref{star Hurewicz} below) were introduced and recently in \cite{KKS-MS} Ko\v{c}inac, Konca and Singh considered some relative versions of them called, respectively, set star Menger, set strongly star Menger, set star Hurewicz and set strongly star Hurewicz properties.

\smallskip
In this paper we study the previous set properties. In particular,  using \cite {BMae}, we easily prove that set strongly star Menger and set strongly star Hurewicz properties are between countable compactness and property of having countable extent. Also we show that the extent of a regular set star Menger or a set star Hurewicz space cannot exceed $\frak c$ and use this result to give a Tychonoff star Menger (star Hurewicz) space which is not set star Menger (set star Hurewicz). In fact, the constructed example (Example \ref{EXAMPLE}) is even star compact and then it gives a positive answer to the following question.

\begin{question}\rm\cite{S-AM} 
	Does there exist a Tychonoff star compact space which is not set star compact?
\end{question}

Moreover, we give a consistent answer (Example \ref{set SM, not set SSM}) to the following question.

\begin{question}\rm\cite{KKS-MS} 
	Does there exist a Tychonoff set star Menger space which is not set strongly star Menger?
\end{question}

Further, we answer in the negative  (Example \ref{product}) to the following

\begin{question}
	\rm \cite{KKS-MS} Is the product of a set star Menger space with a compact space a set star Menger space?
\end{question}

In fact Example \ref{product} shows even more: it proves that set star compact and set star Lindel\"of properties are not preserved in the product with compact spaces. Then, the same example answers in the negative to the following two questions.

\begin{question}
	\rm Is the product of a set star Hurewicz space with a compact space a set star Hurewicz space?
\end{question}

\begin{question}
	\rm \cite{S-AM} Is the product of a set star compact space with a compact space a set star compact space?
\end{question}

\smallskip

Moreover we give partial answers to the following questions.

\begin{question}\label{QSSM}
	\rm \cite{KKS-MS} Is the product of a set strongly star Menger space with a compact space a set strongly star Menger space?
\end{question}

\begin{question}
	\rm  Is the product of a set strongly star Hurewicz space with a compact space  a  set strongly star Hurewicz space?
\end{question}

\smallskip

No separation axiom will be assumed a priori.
Recall that a family of sets is almost disjoint if the intersection
of any two distinct elements is finite. Let $\cal A$ be an almost disjoint family of infinite subsets of
$\omega$. Put $\Psi({\mathcal A})=\omega\cup {\mathcal A}$ and topologize $\Psi({\mathcal A})$ as follows: the points of $\omega$ are isolated
and a basic neighbourhood of a point $a\in\mathcal A$ takes the form $\{a\}\cup (a\setminus F)$, where
$F$ is a finite set. $\Psi({\mathcal A})$ is called Isbell-Mr\' owka or $\Psi$-space (see \cite{Eng}).
Recall that for $f,g\in \omega^\omega$, $f\leq^*g$ means that $f(n)\leq g(n)$ for all but finitely many $n$ (and $f\leq g$ means that $f(n)\leq g(n)$ for all $n\in\omega$). A subset $B\subseteq \omega^\omega$ is bounded if there is $g\in\omega^\omega$ such that $f\leq^*g$ for every $f\in B$. $D\subseteq \omega^\omega$ is cofinal if for each $g\in \omega^\omega$ there is $f\in D$ such that $g\leq^*f$. The minimal cardinality of an unbounded subset of $\omega^\omega$ is denoted by ${\frak b}$, and the minimal cardinality of a cofinal subset of $\omega^\omega$ is denoted by ${\frak d}$. The value of ${\frak d}$ does not change if one considers the relation $\leq$ instead of $\leq^*$ \cite[Theorem 3.6]{vD}.

\section{On set star Menger and set strongly star Menger properties.}

In \cite{K}, Ko\v{c}inac introduced the following star versions of Menger property.

\begin{definition}
	\rm \cite{K}\label{star Menger}
A space $X$ is

$\bullet$ star Menger (briefly, SM) if for each sequence $({\cal U}_n : n \in \omega)$ of open covers of $X$ there exists a sequence $({\cal V}_n : n \in \omega)$  such that ${\mathcal V}_n$, $n\in\omega$, is a finite subset of ${\mathcal U}_n$ and $X=\bigcup_{n \in \omega}st(\bigcup{\mathcal V}_n,{\mathcal U}_n)$;

$\bullet$ strongly star Menger (briefly, SSM) if for each sequence $({\cal U}_n : n \in \omega)$ of open covers of $X$ there exists a sequence $(F_n : n \in \omega)$  such that $F_n$, $n\in\omega$, is a finite subset of $X$ and $X=\bigcup_{n \in \omega}st(F_n,{\mathcal U}_n)$.
\end{definition}

The following result gives a characterization of the SSM property in terms of a relative version of it. 

\begin{proposition}\rm \label{caratterizzazione} The following are equivalent for a space $X$:	\begin{enumerate}
		\item 	
		$X$ is SSM;
		\item 
		
		for each nonempty subset $A$ of $X$ and each sequence $({\cal U}_n : n \in \omega)$ of collection of open sets of $X$ such that $\overline A\subset \bigcup{\mathcal U}_n$ for every $n\in\omega$, there exists a sequence $(F_n : n \in \omega)$  such that $F_n$, $n\in\omega$, is a finite subset of $X$ and $A\subset \bigcup_{n \in \omega}st(F_n,{\mathcal U}_n)$.
	\end{enumerate}

\begin{proof}
$2.\Rightarrow 1.$ is obvious. Let $A \subseteq X$ be a nonempty subset and $({\cal U}_n : n \in \omega)$ be a sequence of families of open sets of $X$ such that $\overline A\subseteq \bigcup {\mathcal{U}}_{n}$ for every $n \in \omega$. Define 
\begin{center}
	${\mathcal{U}}^{'}_{n} = {\mathcal{U}}_{n} \cup \{ X \setminus \overline{A} \}$ 
\end{center}
for all $n \in \omega$. Clearly, each ${\mathcal{U}}^{'}_{n}$ is an open cover for $X$. Since $X$ is SSM, there is a sequence $( F_{n} : n\in \omega)$ of finite subsets of $X$ such that $X=\bigcup_{n \in \omega}st(F_n,{\mathcal U}_n)$. Fix $x \in A$. Then there exists $n \in \omega$ such that $x \in st(F_{n},{\mathcal{U}}^{'}_{n})$. Observe that
\begin{center}
	$x \in st(F_{n},{\mathcal{U}}^{'}_{n}) \Leftrightarrow st(x,{\mathcal{U}}^{'}_{n}) \cap F_{n} \neq \emptyset$
\end{center}
We also have that
\begin{center}
	$st(x,{\mathcal{U}}^{'}_{n}) = \bigcup \{ U \in {\mathcal{U}}^{'}_{n} : x \in U \} = \bigcup \{ U \in {\mathcal{U}}_{n} : x \in U \} = st(x,{\mathcal{U}}_{n})$
\end{center}
So
\begin{center}
	$x \in st(F_{n},{\mathcal{U}}^{'}_{n}) \Leftrightarrow st(x,{\mathcal{U}}^{'}_{n}) \cap F_{n} \neq \emptyset \Leftrightarrow st(x,{\mathcal{U}}_{n}) \cap F_{n} \neq \emptyset \Leftrightarrow x \in st(F_{n},{\mathcal{U}}_{n})$.
\end{center}
Since $x$ is an arbitrary point of $A$, we have that $A\subset \bigcup_{n \in \omega}st(F_n,{\mathcal U}_n)$.
\end{proof}
\end{proposition}
In \cite{KKS-MS} the following relative version of the SM and SSM properties were considered.

\begin{definition}\rm \cite{KKS-MS} A space $X$ is 
	
	$\bullet$ set star Menger (shortly, set SM) if for each nonempty subset $A$ of $X$ and for each sequence $({\cal U}_n : n \in \omega)$ of collection of open sets of $X$ such that $\overline A\subset \bigcup{\mathcal U}_n$ for every $n\in\omega$, there exists a sequence $({\cal V}_n : n \in \omega)$  such that ${\mathcal V}_n$, $n\in\omega$, is a finite subset of ${\mathcal U}_n$ and $A\subset \bigcup_{n \in \omega}st(\bigcup{\mathcal V}_n,{\mathcal U}_n)$.
	
		$\bullet$ set strongly star Menger (shortly, set SSM) if for each nonempty subset $A$ of $X$ and for each sequence $({\cal U}_n : n \in \omega)$ of collection of open sets of $X$ such that $\overline A\subset \bigcup{\mathcal U}_n$ for every $n\in\omega$, there exists a sequence $(F_n : n \in \omega)$  such that $F_n$, $n\in\omega$, is a finite subset of $\overline A$ and $A\subset \bigcup_{n \in \omega}st(F_n,{\mathcal U}_n)$.
\end{definition}

\smallskip

The following result is easy to check.

\begin{proposition}\rm\label{ered}
A space $X$ is set SSM iff every closed subspace of $X$ is SSM.
\end{proposition}

The previous result is not true for set SM spaces as the following example shows.

\begin{example}\rm\label{ex ered}
A set SM space having a closed subspace which is not SM.
\end{example}
Consider the set SM space $X$ of Example \ref{product} below and its closed subspace $A$. Since $A$ is a discrete subspace of uncountable cardinality, it is not SM. $\hfill\triangle$

\smallskip

Recall that in \cite{S2} it is proved that the extent of a $T_1$ SSM space can be arbitrarily big. Also 

\begin{proposition}\rm \label{Sakai}\cite[Corollary 2.2]{Sakai}
	Every closed and discrete subspace of a regular SSM space has cardinality less than $\frak c$. Hence a SSM space has extent less or equal to $\frak c$.
\end{proposition}

It is well known that a CC space has countable extent. Since every CC space is set SSC (see \cite[Theorem 2.1.4]{vDRRT} and recall that CC property is hereditary with respect to closed sets) and every set SSL space has countable extent \cite[Proposition 3.1]{BMae}
we have that
$$\hbox{CC}\Rightarrow \hbox{set SSH}\Rightarrow \hbox{set SSM}\Rightarrow \hbox{countable extent}.$$

Note that the previous implications can not be reversed. Indeed,
of course, every Hurewicz non countably compact space is a set SSH non countably compact space: consider, for example, the discrete space $\omega$. For the converse of the other implications see Examples \ref{exampleSSH} and \ref{example}.

\smallskip

In \cite{MAT} it was shown that the extent of a Tychonoff SSL space can be arbitrarily large
	(note that in \cite{MAT} a SSL space is called a space with countable weak extent). In \cite{Sakai} the space constructed in \cite{MAT} was used to prove that the extent of a Tychonoff SM (in fact SC) space can be arbitrarily large.
Moreover in \cite{Sakai} the author shows the following
\begin{theorem}\rm\cite{Sakai}\label{Sakai}
	If $X$ is a regular SM space such that $w(X)=\frak c$, then every closed and discrete subspace of $X$ has cardinality less than $\frak c$. Hence, we have $e(X)\leq \frak c$.
\end{theorem}
Now we show that 	the extent of a regular set SM space cannot exceed $\frak c$.

\begin{theorem}\rm\label{extentsetSM}\label{Sakai3}
	If $X$ is a regular set SM space, then every closed and discrete subspace of $X$ has cardinality less than $\frak c$. Hence, we have $e(X)\leq \frak c$.
\end{theorem}
\begin{proof} Fix $Y$ a closed and discrete subspace of $X$ and assume $|Y|=\frak c$. Consider a family $\cal B$ of open subsets of $X$ such that for every $y\in Y$ there exists $B\in{\cal B}$ such that $y\in B$ and $\overline{B}\cap Y=\{y\}$ and suppose that $|\cal B|=\frak c$. Denote by $[\cal B]^{<\omega}$ the family of all finite subsets of $\cal B$, by $\mathbb P=([\cal B]^{<\omega})^{\omega}$ the family of all the sequences of elements of $[\cal B]^{<\omega}$ and introduce on $\mathbb{P}$ the partial order "$\leq$" defined as follows: if
$({\cal B}_n')_{n\in\omega},({\cal B}_n'')_{n\in\omega}\in {\mathbb P} \hbox{ then  }({\cal B}_n')_{n\in\omega} \leq ({\cal B}_n'')_{n\in\omega} \hbox{  means  }{\cal B}_n'\subseteq {\cal B}_n'' \hbox{  for every  } n\in\omega$. Let $\{({\cal B}_{\alpha,n})_{n\in\omega}: \alpha<\frak c\}$ be a cofinal family in $(\mathbb P, \leq)$. Take $Z=\{y_\alpha: \alpha <\frak c\}$ by choosing for every $\alpha<\frak c$ a point $y_\alpha \in Y\setminus \bigcup_{n\in\omega}\overline{\bigcup {\cal B}_{\alpha,n}}$ and $y_\alpha\not=y_\beta$ for $\alpha\not=\beta$. For every $\alpha < \frak c$ let $\{V_n(y_\alpha): n\in\omega\}$ be a sequence of open neighbourhoods of $y_\alpha$ such that $V_n(y_\alpha)\subseteq B$ for some $B\in {\cal B}$ and every $n\in\omega$ and

$V_n(y_\alpha)\cap \bigcup{\cal B}_{\alpha,n}=\emptyset$ for every $n\in\omega$. For every $n\in\omega$ put ${\cal U}_n=\{V_n(y_\alpha): \alpha <\frak c\}$. Clearly $ Z=\overline{Z}\subseteq \bigcup {\cal U}_n$ for every $n\in\omega$. We will show that the subset $Z$ and the sequence $({\cal U}_n: n\in\omega)$ do not satisfy the set SM property. Let $({\cal V}_n: n\in\omega ) $ be any sequence of finite subsets of ${\cal U}_n$ for every $n\in\omega$. Let $({\cal B}_n': n\in\omega)\in \mathbb P$ such that every member of ${\cal V}_n$ is contained in a member of ${\cal B}_n'$. Since $\{({\cal B}_{\alpha,n})_{n\in\omega}: \alpha<\frak c\}$ is a cofinal family in $\mathbb P$, there exists $\gamma <\frak c$ such that ${\cal B}_n'\subseteq {\cal B}_{\gamma,n}$ for every $n\in\omega$. Then $V_n(y_\gamma) \cap \bigcup {\cal V}_n \subseteq V_n(y_\gamma) \cap \bigcup {\cal B}_{\gamma,n}=\emptyset$ for every $n\in \omega$. Since $V_n(y_\gamma)$ is the only member of ${\cal U}_n$ containing $y_\gamma$, we have $y_\gamma\not\in\bigcup_{n\in\omega}st(\bigcup {\cal V}_n,{\cal U}_n)$.
\end{proof} 

\smallskip

Example \ref{consistent} below gives a consistent example of a SSM not set SSM space. In fact, such an example was already described in \cite{KKS-MS}; here we show that it can be easily obtained from the next characterization and from the fact that set SSM spaces 
have countable extent.

\begin{theorem}\rm \cite{BMat}\label{BoMat}
	The following are equivalent:
	\begin{itemize}
		\item[(i)] $\Psi({\cal A})$ is SSM
		\item[(ii)] $|{\cal A}| < \mathfrak{d}$.
	\end{itemize}
\end{theorem}

\begin{example}\label{consistent}
	\rm \cite{KKS-MS} ($\omega_1<\frak d$) There exists a SSM not set SSM space.
\end{example}

Assume $\omega_1<\frak d$ and consider $\Psi({\cal A})$ with $|{\cal A}| = \omega_1$. 
By Theorem \ref{BoMat} and since $e(\Psi({\cal A}))>\omega$, we have that $\Psi({\cal A})$ is a SSM not set SSM space. $\hfill\triangle$

\begin{question}\rm Does there exist a ZFC example of a SSM not set SSM space?  
\end{question}

Using Theorem \ref{Sakai3} we can give a Tychonoff space distinguishing SM and set SM properties. In fact, the following example distinguishes SCness and set SCness too.
 
\begin{example}\label{EXAMPLE}
	\rm A Tychonoff SC (hence SM) space which is not set SM (hence not set SC).
\end{example} 
In \cite{MAT}, for each infinite cardinal $\tau$ the following space $X(\tau)$ was considered. Let $Z= \{f_{\alpha} : \alpha < {\frak \tau} \}$ 
	where $f_\alpha$ denotes the points in $2^{\tau}$ with only the $\alpha$th coordinate equal to 1.
	Consider the set $$X(\tau)=(2^{\frak \tau}\times ({\tau}^+ +1))\setminus ((2^{\tau}\setminus Z)\times \{{\tau}^+\})$$
	with the topology inherited from the product topology on $2^{\tau}\times ({\tau}^++1)$.
	Denote $X_0=2^{\tau}\times {\tau}^+$ and $X_1=Z\times{\{\tau}^+\}$. Then $X(\tau)=X_0\cup X_1$. $X_1$ is a closed and discrete subspace of $X(\tau)$ of cardinality $\tau$. So the extent of $X(\tau)$ is $\tau$.\\
	In \cite{Sakai} it is proven that the space $X(\mathfrak{c})$ is SC (hence SM). By Theorem \ref{Sakai3}, $X(\mathfrak{c})$ it is not set SM. $\hfill\triangle$ 

\bigskip

Recall the following result:

\begin{proposition}
 	\rm \cite{Sakai} \label{Sakai1} Every SL (SSL) space of cardinality less than $\frak d$ is SM (SSM).
 \end{proposition}

Now we prove that the set versions of the previous proposition holds.

\begin{proposition}\label{setSL implys setSM}
 	\rm Every set SL space of cardinality less than $\frak d$ is set SM.
\end{proposition}
\begin{proof}
	Let $X$ be a set SL space of cardinality less than $\frak d$. Let $A\subseteq X$ and $({\cal U}_n\,\,:\,\, n\in\omega)$ be a sequence of families of open sets of $X$ such that $\overline{A}\subseteq \bigcup {\cal U}_n$ for every $n\in\omega$. For every $n\in \omega $ there is a countable subfamily ${\cal V}_n=\{V_{n,m}\,\,:\,\,m\in \omega\}$ of ${\cal U}_n$ such that $A\subseteq st(\bigcup {\cal V}_n, {\cal U}_n)$. For every $x\in A$ we choose a function $f_x\in\omega^\omega$ such that $st(x,{\cal U}_n)\cap V_{n, f_x(n)}\not= \emptyset$ for all $n\in\omega$. Since $\{f_x\,\,:\,\,x\in A\}$  is not a cofinal family in $(\omega^\omega, \leq)$, there are some $g\in \omega^\omega$ and $n_x\in\omega$ for $x\in A$ such that $f_x(n_x)<g(n_x)$. Let ${\cal W}_n=\{V_{n,j}\,\,:\,\,j\leq g(n)\}$. Then $A\subseteq \bigcup_{n\in\omega}st(\bigcup {\cal W}_n, {\cal U}_n)$.
\end{proof}

In a similar way we can prove the following

\begin{proposition}\label{frakd}
	\rm Every set SSL space of cardinality less than $\frak d$ is set SSM.
\end{proposition}

Then, by Proposition \ref{Proposition 3.1} we obtain

\begin{corollary}\label{set-SSM}
	\rm For every $T_1$ space $X$ of cardinality less than $\frak d$, the following are equivalent:
\begin{enumerate}
\item $X$ is set SSM 
\item $e(X)=\omega$.
\end{enumerate}
\end{corollary}

\begin{example}\label{example}\rm
There is a Tychonoff space of cardinality $\frak d$ having countable extent  which is not set SSM.
\end{example}
Let $\Bbb P$ the space of irrationals. Take any non-Menger subspace $X\subset\Bbb P$ of cardinality $\frak d$ (for instance, consider the Baire space $\omega^\omega$ which is homeomorphic to $\Bbb P$ and take a cofinal subset of cardinality $\frak d$. It is well known that any cofinal subset of $\omega^\omega$ is not Menger). Of course, $X$ is a paracompact space having countable extent. Since in the class of paracompact Hausdorff spaces we have that M $\Leftrightarrow$ SM
(see \cite{K}), we have that $X$ is not set SSM.    $\hfill\triangle$

\bigskip

By Corollary \ref{set-SSM} and Example \ref{example} we have:

\begin{corollary}\rm
The following statements are equivalent:
\begin{enumerate}
\item $\omega_1<\frak d$;
\item every $T_1$ space of cardinality $\omega_1$ having countable extent is set SSM.
\end{enumerate}
\end{corollary}

\smallskip

Recall the following result.

\begin{theorem}\rm\cite{Sakai}\label{Sakai2}
	The following statements are equivalent for regular spaces.
	\begin{enumerate}
		\item $\omega_1={\frak d}$;
		\item if $X$ is a SSM space, then $e(X)\leq\omega$.
\end{enumerate}
\end{theorem}

Now we prove

\begin{theorem}\rm
	The following statements are equivalent for regular spaces.
		\begin{enumerate}
			\item $\omega_1={\frak d}$;
			\item if $X$ is a SSM space, then $e(X)\leq\omega$;
			\item for spaces of cardinality less than ${\frak d}$,  set SSM and SSM are equivalent properties.
			\item for spaces of cardinality less than ${\frak d}$,  set SSL and SSL are equivalent properties.
			\item every closed subspace of a SSM space  $X$ such that $|X|<{\frak d}$ is SSM.
	\end{enumerate}
\end{theorem}

\begin{proof}
$1.\Leftrightarrow 2.$ holds by Theorem \ref{Sakai2}. Now we prove 
$2.\Rightarrow 3.$. Let $X$ be a space of cardinality less than ${\frak d}$. By 2. and Corollary \ref{set-SSM}, we have that $X$ is SSM iff $X$ is set SSM.
Now we prove $3.\Rightarrow 1.$. Assume $\omega_1<{\frak d}$. Consider a space $\Psi({\cal A})$, with $|{\cal A}|=\omega_1$. By Theorem \ref{BoMat}, $\Psi({\cal A})$ is SSM, and since $e(\Psi({\cal A}))> \omega$, $\Psi({\cal A})$ is not set SSM. $3.\Leftrightarrow 4.$ is obvious. $3.\Leftrightarrow 5.$ follows from Proposition \ref{ered}.
\end{proof}

Of course, countable spaces are Menger, then set SSM and SSM.

\begin{corollary}\rm
	For regular spaces $X$ such that $\omega<|X|<{\frak d}$,  SSM and set SSM are not equivalent properties.
\end{corollary}
\begin{corollary}\rm
	Uncountable regular spaces in which SSM and set SSM are equivalent properties have  cardinality $\geq {\frak d}$.
\end{corollary}

In \cite[Example 5]{KKS-MS} Ko\v{c}inac, Konca and Singh constructed a $T_1$ set SM space which is not set SSM and posed the following question.

\begin{question}\rm\cite{KKS-MS} \label{QUESTION}
	Does there exist a Tychonoff set SM space which is not set SSM?
\end{question}
Using Proposition \ref{setSL implys setSM} we can give a consistent answer to Question \ref{QUESTION}. 

\begin{example}\label{set SM, not set SSM} 
	\rm ($\omega_1<\frak d$)  A Tychonoff set SM space which is not set SSM.
\end{example}
Assume $\omega_1<\frak d$ and consider $\Psi(\mathcal A)$ with $|{\mathcal A}|=\omega_1$. Since $\Psi(\mathcal A)$ is separable, it is set SL hence, by Proposition \ref{setSL implys setSM}, it is set SM. Since $e(\Psi(\mathcal A))>\omega$, $\Psi(\mathcal A)$ is not set SSM. $\hfill\triangle$

\section{On the product of set SM and set SSM with compact spaces.}

Recall that the product of a SC (SSC) space with a compact space is  SC (SSC) (\cite{Fle}, \cite{vDRRT}); further the product of a SL space with a compact space is  SL \cite{vDRRT} while the product of a SSL space with a compact space need not be  SSL \cite[Example 3.3.4]{vDRRT}.
In \cite{K} Ko\v{c}inac proved that the product of a SM space with a compact space is SM. Using \cite[Lemma 2.3]{BM1}, Matveev noted that  assuming $\omega_1<\frak d$,  if $X=\Psi({\cal A})$ with $|{\mathcal A}|=\omega_1$ and $Y$ is a compact space such that $c(Y)>\omega$, then the product $X\times Y$ is not SSL, hence not SSM; therefore he gave a consistent example of a not SSM space which is the product of a SSM space and a compact space.  Then, it is natural to consider the following questions. 

\begin{question}\label{QSSM}
	\rm \cite{KKS-MS} Is the product of a set SSM space with a compact space a  set SSM space?
\end{question}

\begin{question}\label{QSM}
\rm \cite{KKS-MS} Is the product of a set SM space with a compact space a set SM space?
\end{question}

In the following we give a partial answer to Question \ref{QSSM} and a negative answer to Question \ref{QSM}. Note that we also show that set SSL property is preserved in the $T_1$ product with compact spaces and that set SC and set SL properties are not preserved in the product with compact spaces. (For completness, we note that, by Proposition  \ref{prop}, set SSC property is preserved in the Hausdorff product with compact spaces).

\smallskip

The following fact can be easly checked (we give the proof for sake of completeness). Recall that a map is perfect if it is continuous, closed, onto and each fiber is compact.

\begin{proposition}\rm\label{perfect}
If $f:X\to Y$ is a perfect map
and $A$ is an uncountable closed and discrete subspace of $X$, then $f(A)$ is an uncountable closed and discrete subspace of $Y$.
\end{proposition}
\begin{proof}
	Let $f$ and $A$ as in the hypothesis. Clearly $f(A)$ is closed in $Y$. Note that, for every $y\in f(A)$, $f^{-1}(y)\cap A$ is a closed subset of the compact subspace $f^{-1}(y)$ and then, since $A$ is discrete, it is finite. Then $A$ is countable, otherwise 
$f(A)$ is countable.  Now, fix $y\in f(A)$ and say $f^{-1}(y)\cap A=\{x_1,...,x_n\}$. For every $i=1,...,n$ fix an open subset $U_i$ of $X$ such that $A\cap U_i=\{x_i\}$ and put $U=\bigcup_{i=1}^{n}U_i$. Since $A\setminus U$ is a closed subset of $X$, we have that $f(A\setminus U)=f(A)\setminus \{y\}$ is a closed subset of Y, and then $\{y\}$ is open in $f(A)$ with the topology inherited from $Y$.
\end{proof}

By the previous proposition, we obtain the following result.

\begin{corollary}\label{preservationextent}
	\rm The product of a space having countable extent with a compact space has countable extent.
\end{corollary}
\begin{proof} Let $X$ be a space with countable extent and $Y$ be a compact space. The projection from $X\times Y$ onto $X$ is a perfect map. Then, by Proposition \ref{perfect}, $e(X\times Y)=\omega$. \end{proof}

By Proposition \ref{Proposition 3.1}, the previous result can be restated as follows.

\begin{proposition}
	\rm The $T_1$ product of a set SSL space with a compact space is set SSL.
\end{proposition}

\begin{corollary}\label{prodextent}
	\rm The product of a set SSM space with a compact space has countable extent.
\end{corollary}

Then, by  Corollary \ref{set-SSM} we have

\begin{corollary}
	\rm The $T_1$ product of cardinality less than $\frak d$ of a set SSM space with a compact space is set SSM.
		\end{corollary}

\smallskip

Recall the following proposition.

\begin{proposition}\label{prop34}
	\rm \cite[Proposition 3.4]{BMae} Let $X$ be a space. If there exist a closed and discrete subspace $D$ of $X$ having uncountable cardinality and a disjoint family ${\mathcal U}=\{O_a : a\in D\}$ of open neighbourhoods of points $a\in D$, then $X$ is not set SL.
\end{proposition}

Now we prove the following useful result.

\begin{proposition}\label{corretta?}
	\rm
If $e(X)>\omega$ and $c(Y)>\omega$, where $Y$ is $T_1$, then $X\times Y$ is not set SL.
\end{proposition}

\begin{proof}
	Let $S=\{s_\alpha : \alpha <\omega_1\}$ 
	be a closed and discrete subset of $X$,  ${\mathcal O}=\{O_\alpha : \alpha < \omega_1\}$ be a pairwise disjoint family of nonempty open subsets of $Y$. For every $\alpha < \omega_1$, fix $t_\alpha\in O_\alpha$. Put $A=\{(s_\alpha,t_\alpha): \alpha < \omega_1\}$. It is obvious that $A$ is an uncountable discrete subspace of $X\times Y$.
	Now we prove that $A$ is closed. For every $\alpha < \omega_1$ there exists an open set, say $N_\alpha$, such that $N_\alpha\cap S=\{s_\alpha\}$. 
	Then $(X\times Y)\setminus A=((X\setminus S)\times Y) \cup \bigcup_{\alpha < \omega_1} (N_\alpha \times (Y\setminus \{t_\alpha\}))$. Then, by Proposition \ref{prop34}, $X\times Y $ is not set SL. 
\end{proof}

\begin{example}\rm \label{product}  There exists a set SC (hence set SL and  set SM) space $X$ and a compact space $Y$ with $c(Y)>\omega$ such that $X\times Y$ is not set SL  (hence neither set SM nor set SC). 
\end{example}
Consider the set $X=\omega_1 \cup A$, where $A=\{a_\alpha : \alpha\in \omega_1\}$ is a set of cardinality $\omega_1$, topologized as follows: $\omega_1$ has the usual order topology and is an open subspace of $X$; a basic neighborhood of a point $a_\alpha\in A$ takes the form 
$$O_\beta(a_\alpha)=\{a_\alpha\}\cup (\beta, \omega_1), \hbox{ where } \beta<\omega_1.$$ In \cite{BMae} it was proved that $X$ is set SC, hence $X$ is set SM. We have that $e(X)>\omega$. If $Y$ is any compact space with $c(Y)>\omega$, by Proposition \ref{corretta?}, $X\times Y$  is not set SL. 
$\hfill\triangle$

\section{On set star Hurewicz and set strongly star Hurewicz properties.}

Recall the following definitions.

\begin{definition}
	\rm \cite{K, BCK}\label{star Hurewicz}
A space $X$ is

$\bullet$ star Hurewicz (briefly, SH) if for each sequence $({\cal U}_n : n \in \Bbb N)$ of open covers of $X$ there exists a sequence $({\cal V}_n : n \in \Bbb N)$  such that ${\mathcal V}_n$, $n\in\omega$, is a finite subset of ${\mathcal U}_n$ and $\forall x \in X$, $x \in st(\bigcup {\mathcal{V}}_{n},{\mathcal{U}}_{n})$ for all but finitely many $n \in \omega$;

$\bullet$ strongly star Hurewicz (briefly, SSH) if for each sequence $({\cal U}_n : n \in \Bbb N)$ of open covers of $X$ there exists a sequence $(F_n : n \in \Bbb N)$  such that $F_n$, $n\in\omega$, is a finite subset of $X$ and $\forall x \in X$, $x \in st(F_{n},{\mathcal{U}}_{n})$ for all but finitely many $n \in \omega$.
\end{definition}

The following result is a characterization of SSH property in terms of a relative version of it. The proof is similar to the proof of Proposition \ref{caratterizzazione}.

\begin{proposition}\rm The following are equivalent for a space $X$:
	\begin{enumerate}
		\item 	
		$X$ is SSH;
		\item 
		for each nonempty subset $A$ of $X$ and for each sequence $({\cal U}_n : n \in \Bbb N)$ of collection of open sets of $X$ such that $\overline A\subset \bigcup{\mathcal U}_n$ for every $n\in\omega$, there exists a sequence $(F_n : n \in \Bbb N)$  such that $F_n$, $n\in\omega$, is a finite subset of $X$ and $\forall x \in A$, $x \in st(F_{n},{\mathcal{U}}_{n})$ for all but finitely many $n \in \omega$.
	\end{enumerate}
\end{proposition}

\begin{definition}\rm\cite{KKS-MS} 
	A space $X$ is
	
	\begin{itemize}
				\item 
		set star Hurewicz (briefly, set SH) if for each nonempty subset $A \subseteq X$ and for each sequence $( {\mathcal{U}}_{n}: n \in \omega)$ of collection of open sets of $X$ such that $\overline A\subseteq \bigcup {\mathcal{U}}_{n}$ for every $n \in \omega$, there exists		
		 a sequence $({\cal V}_n : n \in \Bbb N)$  such that ${\mathcal V}_n$, $n\in\omega$, is a finite subset of ${\mathcal U}_n$ and $\forall x \in A$, $x \in st(\bigcup {\mathcal{V}}_{n},{\mathcal{U}}_{n})$ for all but finitely many $n \in \omega$.
		\item 
		set strongly star Hurewicz (briefly, set SSH) if for each nonempty subset $A \subseteq X$ and for each sequence $( {\mathcal{U}}_{n}: n \in \omega)$ of collection of open sets of $X$ such that $\overline A\subseteq \bigcup {\mathcal{U}}_{n}$ for every $n \in \omega$, there exists		
		 a sequence $( F_{n}: n \in \omega)$ such that $F_n$, $n\in\omega$, is a finite subset of $\overline{A}$ and $\forall x \in A$, $x \in st(F_{n},{\mathcal{U}}_{n})$ for all but finitely many $n \in \omega$.
				\end{itemize}
	\end{definition}

\begin{example}\label{exampleSSH}\rm ($\frak b<\frak d$)
There is a Tychonoff set SSM space which is not set SSH.
\end{example}
Consider an unbounded subset $X$ of the Baire space $\omega^\omega$ of cardinality $\frak b$. Then $X$ is not Hurewicz and, by Corollary \ref{set-SSM}, $X$ is set SSM. Since $X$ is a paracompact space and in the class of paracompact Hausdorff spaces we have that H $\Leftrightarrow $ SH 
(see \cite{BCK}), we have that $X$ is not set SSH.    $\hfill\triangle$

\smallskip

Recall the following characterization of SSH spaces.

\begin{theorem}\rm \cite{BMat}\label{BMat}
	The following properties are equivalent:
	\begin{itemize}
		\item[(i)] $\Psi({\cal A})$ is SSH
		\item[(ii)] $|{\cal A}| < \mathfrak{b}$.
	\end{itemize}
\end{theorem}

Then, we can "easily" give the following result (the same example was given in \cite{Sing3} using a longer proof).

\begin{example}
	\rm ($\omega_1<\frak b$) There exists a SSH not set SSH space.
\end{example}

Assume $\omega_1<\frak b$ and consider $\Psi({\cal A})$ with $|{\cal A}| = \omega_1$. Then, by Theorem \ref{BMat} and since $e(\Psi({\cal A}))>\omega$, we have that $\Psi({\cal A})$ is SSH not a set SSH space. $\hfill\triangle$

\begin{question}\rm Does there exist a ZFC example of a SSH not set SSH space? 
	\end{question}

By Theorem \ref{Sakai3} we can give the following
\begin{theorem}\label{Sakai4}\rm
	If $X$ is a regular set SH space , then every closed and discrete subspace of $X$ has cardinality less than $\frak c$. Hence, we have $e(X)\leq \frak c$.
\end{theorem}

In \cite[Esempio 2.4]{Sing3} it is given a Hausdorff SH space which is not set SH. Now we can provide the following

\begin{example}\label{EX}
		\rm A Tychonoff SC (hence SH) space which is not set SH.
\end{example} 
Consider the space $X(\frak c)$ of Example \ref{EXAMPLE}. $X(\frak c)$ is SC (hence SH) and, by Theorem \ref{Sakai4}, it is not set SH. $\hfill\triangle$ 

\smallskip

Recall the following

\begin{proposition}
 	\rm \cite[Corollary 3.10]{CASASDELAROSA2019572} Every SL (SSL) space of cardinality less than $\frak b$ is SH (SSH).
 \end{proposition}

In analogy to Proposition \ref{setSL implys setSM} and Proposition \ref{frakd}, we can prove the following

\begin{proposition}\label{setSL implys setSH}
 	\rm Every set SL (set SSL) space of cardinality less than $\frak b$ is set SH (set SSH).
 \end{proposition}

\begin{proof}
	Let $X$ be a set SL space of cardinality less than $\frak b$ (the proof is similar if $X$ is set SSL). Let $A\subseteq X$ and $({\cal U}_n\,\,:\,\, n\in\omega)$ be a sequence of families of open sets of $X$ such that $\overline{A}\subseteq \bigcup {\cal U}_n$ for every $n\in\omega$. For every $n\in \omega $ there is a countable subfamily ${\cal V}_n=\{V_{n,m}\,\,:\,\,m\in \omega\}$ of ${\cal U}_n$ such that $A\subseteq st(\bigcup {\cal V}_n, {\cal U}_n)$. For every $x\in A$ we choose a function $f_x\in\omega^\omega$ such that $st(x,{\cal U}_n)\cap V_{n, f_x(n)}\not= \emptyset$ for all $n\in\omega$. Since $\{f_x\,\,:\,\,x\in A\}$  is a bounded family in $(\omega^\omega, \leq^*)$, there exists $g\in \omega^\omega$ such that for every $x\in A$ we have that $f_x(n)\leq g(n)$ for every but finitely many $n\in\omega$. Let ${\cal W}_n=\{V_{n,j}\,\,:\,\,j\leq g(n)\}$. Then for every $x\in A$ we have that $x\in st(\bigcup {\cal W}_n, {\cal U}_n)$ for all but finitely many $n\in\omega$.
\end{proof}

Then, by Proposition \ref{Proposition 3.1}, we have

\begin{corollary}\label{set-SSH}
	\rm For every $T_1$ space $X$ of cardinality less than $\frak b$, the following are equivalent:
\begin{enumerate}
\item $X$ is set SSH 
\item $e(X)=\omega$.
\end{enumerate}
\end{corollary}

By Corollary \ref{set-SSM} and Corollary \ref{set-SSH} we have the following

\begin{corollary}\rm\label{SSHextent}
	For spaces $X$ such that $|X|<{\frak b}$,  the following are equivalent:
\begin{enumerate}
\item $X$ is set SSM
\item $X$ is set SSH
\item $e(X)=\omega$.
\end{enumerate}
\end{corollary}

In \cite{Sing3} the authors give a $T_1$ set SH space which is not set SSH. Now we provide the following 
\begin{example}\label{set SH, not set SSH} 
	\rm($\omega_1<\frak b$)  A Tychonoff set SH space which is not set SSH.
\end{example}
Assume $\omega_1<\frak b$ and consider $\Psi(\mathcal A)$ with $|{\mathcal A}|=\omega_1$. Since $\Psi(\mathcal A)$ is separable, it is set SL hence, by Proposition \ref{setSL implys setSH}, it is set SH. Since $e(\Psi(\mathcal A))>\omega$, $\Psi(\mathcal A)$ is not set SSH. $\hfill\triangle$

\bigskip

Using Example \ref{product} we can show that
\begin{proposition}
	\rm  Set SH property is not preserved in the product with compact spaces.
	\end{proposition}

By Corollary \ref{preservationextent} we have that

\begin{proposition}
	\rm The product of a set SSH space with a compact space has countable extent.
	\end{proposition}

Then, by Corollary \ref{SSHextent} we obtain

\begin{proposition}
	\rm The $T_1$ product of cardinality less than $\frak b$ of a set SSH space with a compact space is set SSH.
	\end{proposition}

The following question is open.

\begin{question}
\rm  Is the product of a set SSH space with a compact space  a set SSH space?
\end{question}

\smallskip
We give the following useful diagram.

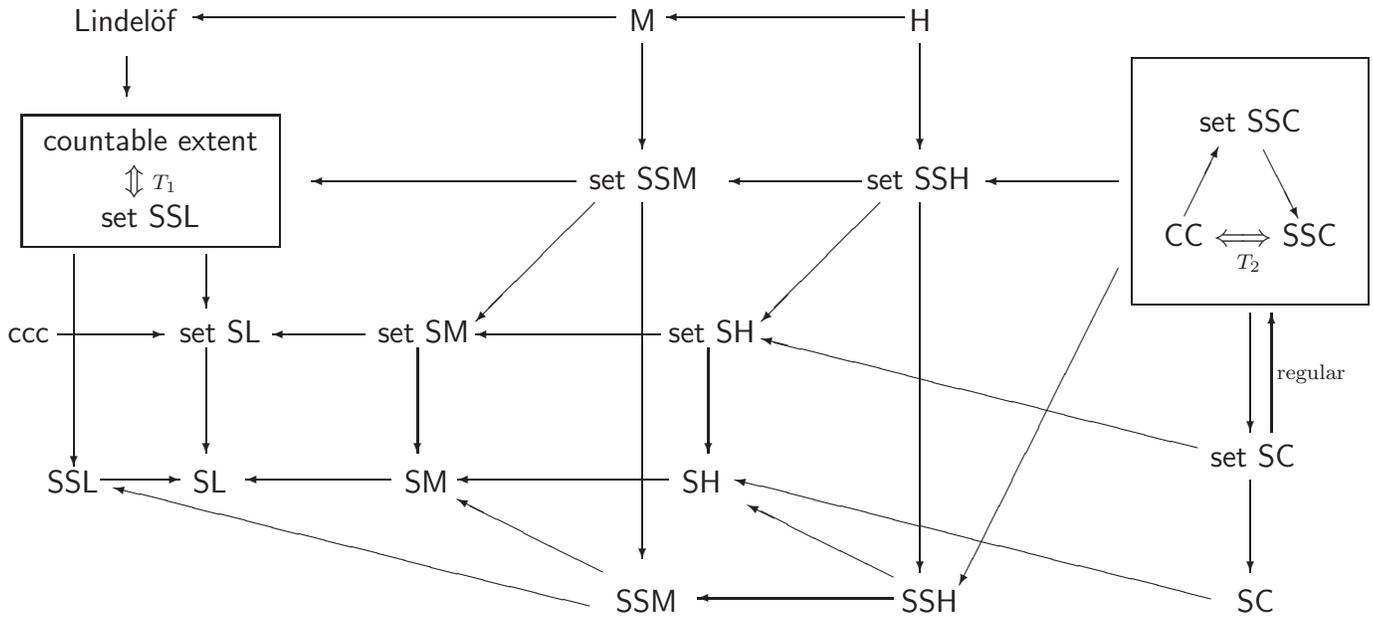
\begin{picture}(150,130)
\put(-50,110){{\sf Lindel\"of}}

\put(-30,100){\vector(0,-1){15}}

\put(-70,50){{\sf $$\boxed{ \begin{array}{c}
			\text{countable extent}\\
			\Updownarrow \text{{\scriptsize $T_1$}}\\
			\text{set SSL}\\
			\end{array}
		}$$}}

\put(-50,25){\vector(0,-1){80}}
\put(-60,-65){{\sf SSL}}

\put(0,25){\vector(0,-1){20}}
\put(-10,-8){{\sf set SL}}

\put(0,-10){\vector(0,-1){40}}
\put(-5,-65){{\sf SL}}
\put(-40,-60){\vector(1,0){30}}

\put(-56,-5){\vector(1,0){40}}
\put(-75,-8){{\sf ccc}}

\put(70,-60){\vector(-1,0){55}}
\put(75,-65){{\sf SM}}
\put(80,-10){\vector(0,-1){40}}
\put(65,-8){{\sf set SM}}
\put(60,-5){\vector(-1,0){35}}

\put(150,-95){\vector(-2,1){55}}
\put(145,-108){\vector(-4,1){180}}
\put(155,-110){{\sf SSM}}
\put(165,45){\vector(0,-1){135}}
\put(145,50){{\sf set SSM}}
\put(147,45){\vector(-1,-1){45}}
\put(140,53){\vector(-1,0){100}}
\put(165,105){\vector(0,-1){40}}
\put(160,110){{\sf M}}
\put(155,115){\vector(-1,0){160}}
\put(350,50){{\sf $$\boxed{ \begin{array}{c}
			\\
			\text{set SSC}\\
			\\
			\\
			\text{ CC $\Longleftrightarrow$ SSC }\\
			\\
			\end{array}
		}$$}}

\put(370,40){\vector(1,2){13}}
\put(400,65){\vector(1,-2){13}}
\put(390,20){\sf {\scriptsize $T_2$}}
\put(345,20){\vector(-1,-2){60}}
\put(345,53){\vector(-1,0){50}}
\put(395,3){\vector(0,-1){45}}
\put(380,-55){{\sf set SC}}
\put(403,-42){\vector(0,1){45}}
\put(405,-22){\sf {\scriptsize regular}}
\put(375,-48){\vector(-4,1){165}}
\put(380,-105){\vector(-4,1){180}}
\put(395,-60){\vector(0,-1){35}}
\put(390,-110){{\sf SC}}
%parte degli Hurewicz
\put(175,-8){{\sf set SH}}
\put(172,-5){\vector(-1,0){70}}
\put(190,-10){\vector(0,-1){40}}
\put(180,-65){{\sf SH}}
\put(175,-60){\vector(-1,0){80}}
\put(263,-110){{\sf SSH}}
\put(260,-97){\vector(-2,1){55}}
\put(260,-105){\vector(-1,0){74}}
\put(266,110){{\sf H}}
\put(264,115){\vector(-1,0){90}}
\put(250,50){{\sf set SSH}}
\put(255,45){\vector(-1,-1){45}}
\put(248,53){\vector(-1,0){50}}
\put(270,105){\vector(0,-1){40}}
\put(270,45){\vector(0,-1){140}}

\end{picture}

\vspace{5cm}

{\bf Acknowledgements.} The authors express gratitude to Masami Sakai for useful suggestions.

\end{document}